\documentclass[12pt]{article}

\usepackage{amssymb}
\usepackage{amsthm}
\usepackage{graphicx, gastex}
\usepackage{amsmath}
\usepackage{latexsym}
\usepackage{amsfonts}

\DeclareSymbolFont{rsfscript}{OMS}{rsfs}{m}{n}
\DeclareSymbolFontAlphabet{\mathrsfs}{rsfscript}

\DeclareSymbolFont{rsfscript}{OMS}{rsfs}{m}{n}
\newtheorem{theorem}{Theorem}
\newtheorem{prop}{Proposition}
\newtheorem{defn}{Definition}
\newtheorem{lemma}{Lemma}
\newtheorem{cor}{Corollary}

\newtheorem{rem}{Remark}

\newcommand{\sch}{{Sch\"{u}tzenberger }}

\newcommand{\rf}{\rightarrow}

\newcommand{\la}{\langle}
\newcommand{\ra}{\rangle}
\newcommand{\wt}{\widetilde}
\newcommand{\inv}{^{-1}}

\def\mapright#1{\smash{\mathop{\longrightarrow}\limits^{#1}}}
\def\vlongrightarrow{\relbar\joinrel\longrightarrow}
\def\vvlongrightarrow{\relbar\joinrel\vlongrightarrow}

\def\longmapright#1{\smash{\mathop{\vlongrightarrow}\limits^{#1}}}
\def\vlongmapright#1{\smash{\mathop{\vvlongrightarrow}\limits^{#1}}}

\title{Amalgams of inverse semigroups and reversible two-counter machines.}
\author{
        Emanuele Rodaro, Pedro V. Silva \\
        Centro de Matem\'{a}tica, Faculdade de Ci\^{e}ncias\\
        Universidade do Porto, 4169-007 Porto, Portugal\\
        \{emanuele.rodaro,pvsilva\}@fc.up.pt
}
\date{\today}

\begin{document}
\maketitle

\begin{abstract}
We show that the word problem for an amalgam
$[S_1,S_2;U,\omega_1,\omega_2]$ of inverse
semigroups may be undecidable even if we assume $S_1$ and $S_2$ (and
therefore $U$) to have finite $\mathcal{R}$-classes and
$\omega_1,\omega_2$ to be
computable functions, interrupting a series of positive decidability
results on the subject. This is achieved by encoding into an appropriate
amalgam of inverse
semigroups $2$-counter machines with sufficient universality, and
relating the nature of certain \sch graphs
to sequences of computations in the machine.
\end{abstract}

\section{Introduction}

If $S_1$, $S_2$ and $U$ are semigroups (groups) such that $U$ embeds
into $S_1, S_2$ via
two monomorphisms $\omega_1,\omega_2$ then
$[S_1,S_2;U,\omega_1,\omega_2]$ is called an
amalgam of semigroups (groups). The amalgamated free product $S_1*_U S_2$
associated with this
amalgam in the category of
semigroups (groups) is defined by the usual universal diagram.
The amalgam $[S_1,S_2;U,\omega_1,\omega_2]$ is said to be strongly
embeddable in a
semigroup (group) $S$ if there exist injective homomorphisms $\phi_i:
S_i\rf S$ such that
$\phi_1|_{\omega_1(U)}=\phi_2|_{\omega_2(U)}$ and $\phi_1(S_1)\cap \phi_2(S_2)=
\phi_1(\omega_1(U))= \phi_2(\omega_2(U))$. It is well known that every
amalgam of groups
embeds in a group (and hence in the amalgamated free product of the
group amalgam)
\cite{lyndon}. However there are many examples showing that in the
category of semigroups
not every amalgam of two semigroups is embeddable into a semigroup,
see \cite{Kim,Sapir}.
On the other hand, T.E. Hall \cite{Hall2} showed that every amalgam of
inverse semigroups
(in the category of inverse semigroups) embeds in an inverse
semigroup, and so in the
corresponding amalgamated free product in the category of inverse semigroups.
In general the class of embeddable amalgams of infinite semigroups
behaves badly from the
algorithmic point of view. In the paper \cite{BiMaMi}, Birget,
Margolis, and Meakin
proved that the amalgamated free product of two finitely presented
semigroups with
solvable word problems and a nice common subsemigroup may have
undecidable word problem.
This result has been further
strengthened by Sapir \cite{Sapir} by showing that an amalgamated free
product of finite
semigroups may have undecidable word problem. These results are in
contrast to the
situation for amalgamated free products of groups where the word problem is
decidable (see
\cite{lyndon}) assuming general nice conditions on the amalgam. The
decidability result
is a consequence of a normal form theorem for the free product with
amalgamation of groups. A
sort of normal form, but with a more geometric flavor, can be defined
in the case of
lower bounded or finite amalgams of inverse semigroups, see
\cite{Ben,Finite}. In
\cite{Ben} Bennett showed that the word problem for lower bounded
amalgam of inverse
semigroups has decidable word problem, this is achieved by giving an
ordered way to build
the \sch automata in such structure. Taking inspiration from Bennett's paper,
Cherubini, Meakin and Piochi in \cite{Finite}, quite in contrast with
Sapir's result,
showed that the word problem for any amalgamated free product of finite inverse
semigroups is decidable in the category of inverse semigroups. As
already pointed out,
this result has been achieved adapting nontrivially the work of Bennett
to the finite case
to obtain a sort of geometric normal form of a word $w$, called by the
authors $Core(w)$.
This fact, along with other lifting properties of \sch automata for
amalgams of finite
inverse semigroups, gave rise to a series of papers
\cite{ChNuRo,RoByci,RoChe} in the
decidability direction. Although these clues can push toward a
decidability result for
amalgams of inverse semigroups with nice conditions, in this paper we
give a result which goes in the opposite direction:

\begin{theorem}
The word problem for an amalgam
$[S_1,S_2;U,\omega_1,\omega_2]$ of inverse
semigroups may be undecidable even if we assume $S_1$ and $S_2$ (and
therefore $U$) to have finite $\mathcal{R}$-classes and
$\omega_1,\omega_2$ to be
computable functions.
\end{theorem}

This is achieved by encoding into an appropriate amalgam of inverse
semigroups a $2$-counter machine having the right property, but still
general enough to hold universality properties with respect to Turing
machines. Then we show how the nature of certain \sch graphs relates
to sequences of computations in the machine.

\section{Inverse semigroups}

A semigroup $S$ is said to be {\em inverse} if, for
each element $a \in S$, there is a unique element $a^{-1} \in S$
such that $a = a a^{-1} a$ and $a^{-1} = a^{-1} a a^{-1}$. A
consequence of the definition is that idempotents commute in any
inverse semigroup. Moreover, a natural partial order can be
defined on $S$ by $a \leq b$ if and only if $a = e b$ for some idempotent
$e$ of $S$.

Inverse semigroups may be regarded as semigroups of
partial one-to-one transformations closed under inversion, so they
arise very naturally in several areas of mathematics. More recently,
also computer scientists have been paying attention to inverse
semigroups for different reasons. First, the inverse of an element in
an inverse semigroup can be seen as the ``undo'' of the action
represented by the
element. Second, algorithmic problems for inverse semigroups have
received considerable attention in the literature during the past
30-35 years and in the cases where such problems are decidable,
the analysis of the complexity of these algorithms is a quite
natural issue.

The natural geometric framework to deal with algorithmic problems in
an inverse semigroup generated by $X$ is
the class of \emph{inverse $X$-graphs}. Write
 $\wt{X} = X\cup X^{-1}$, where $X^{-1}=\{x^{-1}:x\in
X\}$ is a set of formal inverses of $X$. We can extend $\cdot^{-1}$ to
an involution on $\wt{X}^*$ through $(x\inv)\inv = x$ $(x \in X)$ and
$(uv)\inv = v\inv u\inv$ $(u,v \in \wt{X}^*)$.

An inverse $X$-graph is a pair $\Gamma = (V,E)$ where $V = V(\Gamma)$
denotes the
vertex set and $E = E(\Gamma) \subseteq V \times X \times V$ the edge
set. We say
that $\Gamma$ is:
\begin{itemize}
\item
{\em involutive} if $(p,x,q) \in E \Rightarrow (q,x\inv,p) \in E$ for
all $p,q \in V$ and $x \in \wt{X}$;
\item
{\em deterministic} if $(p,x,q), (p,x,r) \in E \Rightarrow q = r$ for
all $p,q,r \in V$ and $x \in \wt{X}$;
\item
{\em connected} if any two vertices can be connected through a path in
$\Gamma$;
\item
{\em inverse} if it is involutive, deterministic and connected.
\end{itemize}
An \emph{inverse $X$-automaton} is a triple
$\mathcal{A} = (\alpha,\Gamma,\beta)$ where $\Gamma$ is an inverse $X$-graph and
$\alpha,\beta\in V(\Gamma)$ are respectively the initial and final
state of the automaton.

The free inverse semigroup on a set $X$, denoted by $FIS(X)$, is the
quotient of the free semigroup $\wt{X}^+$ by the least
congruence $\rho$ that makes the resulting quotient semigroup
inverse (Vagner congruence). Its elements may be seen as finite birooted
inverse $X$-trees usually known as {\em Munn trees} (see
\cite{munn}).

An inverse semigroup presentation is a formal expression of the form
Inv$\la X \mid R \ra$, where $R\subseteq \wt{X}^+ \times
\wt{X}^+$.  The presentation is finite if both $X$ and $R$ are finite.
The inverse semigroup defined by this presentation is the
quotient of $\wt{X}^+$ by the congruence $\tau$ generated by $\rho \cup
R$. We often represent the elements of $R$ as equalities of the form $w_1=w_2$.

In the paper \cite{Steph}, Stephen has extended the notion of Munn
tree introducing the Sch\"utzenberger automaton $\mathcal{A}(X,R;w)$
for a word $w \in \wt{A}^+$ relative to the presentation Inv$\la
X\mid R\ra$ of $S$. The underlying graphs $S\Gamma(X,R;w)$ of these
automata are the strongly connected components of the Cayley graph of
the presentation, which has $S$ as vertex set and edges
$(s,x,s(x\tau))$ $(s \in S, x \in \wt{X}$. More precisely,
$S\Gamma(X,R;w)$ is the restriction of this Cayley graph to the Green
$\mathcal{R}$-class of $w\tau$ in $S$. It turns out that $S\Gamma
(X,R;w)$ is an inverse $X$-graph and we can now define
the \sch automaton
$$\mathcal{A}(X,T;w) = ((ww^{-1})\tau,S\Gamma (X,T;w),w\tau).$$
The
importance of these automata stems from the fact that any two words $w,
w' \in \wt{X}^+$, represent the same element of $S$ if
and only if $\mathcal{A}(X,R;w) = \mathcal{A}(X,R;w')$, equivalent
also to
$$
w \in L(\mathcal{A}(X,R;w')) \quad \mbox{and} \quad  w' \in
L(\mathcal{A}(X,R;w)).
$$
Hence \sch automata are crucial for dealing with algorithmic
problems. In \cite{Steph} it is also provided an
iterative procedure for approximating the \sch automaton relative to a
given presentation Inv$\langle X\mid R\rangle$ of a word $w\in \wt{X}^+$
via two operations, $R$-\emph{expansions}
and \emph{foldings}. Let $\mathcal{A}$ be a finite inverse
$X$-automaton. An $R$-expansion of $\mathcal{A}$ consists in adding a
nonexisting path $p \mapright{r'} q$ to $\mathcal{A}$ (and the corresponding
inverse edges) whenever $p \mapright{r} q$ is a path in $\mathcal{A}$
but no path $p \mapright{r'} q$ is present $(r,r')\in R\cup R^{-1}$. A
folding consists in identifying
two edges starting at the same vertex and labeled by the same
letter of the alphabet $\wt{X}$. We say that $\mathcal{A}$ is
$R$-\emph{closed} if no $R$-expansion can be
performed. The iterative sistematic
application of $R$-expansions (considering all possible instances
simultaneously) and complete folding to the Munn tree of a word $w\in
\wt{X}^+$ (obtained
itself as the complete folding of the linear automaton $\to \bullet
\mapright{w} \bullet \to$) produces a (possibly infinite) sequence of
finite inverse $X$-automata
$$
\mathcal{A}_1\rf\mathcal{A}_2\rf\ldots\rf\mathcal{A}_j\rf\ldots
$$
whose colimit is $\mathcal{A}(X,R;w)$. This colimit procedure can in
fact be
applied starting from any finite inverse $X$-automaton $\mathcal{A}$,
producing as its colimit the $R$-\emph{closure} of
$\mathcal{A}$, i.e., the smallest (with respect to language)
$R$-closed $X$-inverse automaton recognizing all words in $L(\mathcal{A})$
(see \cite{Steph98}). Any finite inverse $X$-automaton obtained from
the Munn tree of $w$ by $R$-expansions and foldings is said to be a
{\em finite approximation} of $\mathcal{A}(X,R;w)$.

We close this section by fixing some notation for amalgams to be
preserved through the whole paper. Let $S_i = \wt{X_i}^+ / \tau_i$ be
the inverse semigroup
defined by Inv$\langle X_i\mid R_i\rangle$ for $i=1,2$ with $X_1\cap
X_2=\emptyset$. Let $\lambda_i:S_i\rf \wt{X_i}^+$ be a map such that, for
each $s\in S_i$, $(\lambda_i(s))\tau_i =s$ $(i = 1,2)$. Finally, let
$U$ be an inverse semigroup and $\omega_i:U \to S_i$ a monomorphism
for $i = 1,2$.
The {\em free product with amalgamation} $S_1*_US_2$ associated to the
amalgam $[S_1,S_2;U,\omega_1,\omega_2]$ is the
inverse semigroup presented by
$$
Inv\la X_1\cup X_2 \mid R_1\cup R_2\cup R_3\ra,
$$
where $R_3=\{(\lambda_1(\omega_1(u)),\lambda_2(\omega_2(u)))\mid u\in
U\}$.

\section{Counter machines}\label{sec:introduction counter machine}
\begin{defn}
  A $k$-counter machine (for short, $CM(k)$) is a system
  $$
  \mathcal{M}=(Q,\delta,\iota,f)
  $$
  where $k$ is the number of tapes, $Q$ is the nonempty finite set of internal
  states, $\iota\in Q$ is the initial state, and $f\in Q$ is the final
  (halting) state. The machine $\mathcal{M}$ uses $A = \{\perp,a\}$ as a
  tape alphabet
  ($\perp$ is a blank symbol), $\delta$ is a move relation which is a
  subset of $(Q \times \{1,\ldots,k\} \times A \times Q)\cup
  (Q \times\{1,\ldots,k\} \times D \times Q)$ where $D = \{ -,0,+\}$
  and the
  symbols $-,0,+$ denote respectively left-shift, no-shift, and
  right-shift of a head of the machine. Tapes are one-way (rightward)
  infinite. The leftmost squares of the tapes contain the blank symbol
  $\perp$, and all the other squares contain the symbol $a$.
\end{defn}
Each element of $\delta$ is thus a quadruple of one of the two forms:
$$
(q,i,s,q'),\;(q,i,d,q')
$$
where $q,q'\in Q$, $i\in\{1,\ldots,k\}$, $s\in A$ and
$d\in D$. A quadruple of the form $(q,i,s,q')$ means that if
$\mathcal{M}$ is in the state $q$ and the $i$th-head is reading the
symbol $s$ then the machine changes its state into $q'$. This
instruction is used
to test whether the content of a counter is zero (the head is reading
the symbol $\perp$) or positive (the head is reading a square with
symbol $a$). We call this kind of instructions \emph{test
  instructions}.
On the other hand an instruction $(q,i,d,q')$ means that if
$\mathcal{M}$ is in the state $q$ then:
\begin{itemize}
\item
it shifts the $i$th-head one
cell to the right if $d = +$,
\item
it shifts the $i$th-head one
cell to the left if $d = -$,
\item
it keeps the $i$th-head at the same
cell if $d = 0$,
\end{itemize}
and changes its state into $q'$. The evolution
of a $CM(k)$ $\mathcal{M}$ can be followed through instantaneous
descriptions of the machine:

\begin{defn}
An instantaneous description (for short, ID) of a $CM(k)$
$\mathcal{M}=(Q,\delta,\iota,f)$ is a $(k+1)$-tuple $(q,n_1,\ldots,
n_k)\in Q\times\mathbb{N}^k$. It represents that $\mathcal{M}$ is in
state $q$ and the $i$th-head is in position
$n_i$ for $i=1,\ldots,k$, where we assume the position of the leftmost
square of a tape to be $0$. The transition relation
$\vdash_{\mathcal{M}}$ is defined as follows:
$$
(q,n_l,\ldots, n_{i-1},n_i,n_{i+1}\ldots,
n_k)\vdash_{\mathcal{M}}(q',n_l,\ldots, n_{i-1},n_i',n_{i+1}\ldots,
n_k)
$$
holds if one of the following conditions is satisfied:
\begin{enumerate}
  \item $(q,i,\perp,q')\in\delta$ and $n_i=n_i'=0$.
  \item $(q,i,a,q')\in\delta$ and $n_i=n_i'>0$.
  \item $(q,i,-,q')\in\delta$ and $n_i-1=n_i'$.
  \item $(q,i,0,q')\in\delta$ and $n_i=n_i'$.
  \item $(q,i,+,q')\in\delta$ and $n_i+1=n_i'$.
\end{enumerate}
We denote the reflexive and transitive closure of
$\vdash_{\mathcal{M}}$ by $\vdash_{\mathcal{M}}^*$ and the $n$-step
transition relation by $\vdash_{\mathcal{M}}^n$.
\end{defn}

\begin{defn}
Let $\mathcal{M}=(Q,\delta,\iota,f)$ be a $CM(k)$.
We say that the $k$-tuple $(n_1,\ldots,n_k) \in \mathbb{N}^k$ is accepted by
$\mathcal{M}$ if
$$(\iota,n_1,\ldots,n_k) \vdash_{\mathcal{M}}^* (f,n'_1,\ldots,n'_k)$$
for some $(n'_1,\ldots,n'_k) \in \mathbb{N}^k$.
\end{defn}

\begin{defn}
Let $\mathcal{M}=(Q,\delta,\iota,f)$ be a $CM(k)$ and
$\alpha_i=(p_i,j_i,x_i,p_i')$ with $i\in\{1,2\}$ be two distinct
quadruples in $\delta$. We say that $\alpha_1,\alpha_2$ overlap in
domain if
$$
p_1=p_2\wedge (j_1\neq j_2\vee x_1=x_2\vee x_1\in D\vee x_2\in D).
$$
We say that $\alpha_1,\alpha_2$ overlap in range if
$$
p_1'=p_2'\wedge (i_1\neq i_2\vee x_1=x_2\vee x_1\in D\vee x_2\in D).
$$
\end{defn}
A quadruple $\alpha \in \delta$ is called \emph{deterministic}
(\emph{reversible}
respectively) if there is no other quadruple in $\delta$ which
overlaps in domain (range) with $\alpha$. The machine $\mathcal{M}$
is called \emph{deterministic} (\emph{reversible} respectively) if
every quadruple in $\delta$ is deterministic (reversible). It is clear
from the definitions that every ID of a deterministic (reversible,
respectively) $CM(k)$ has at most one ID that immediately follows
(precedes) it in some computation. In this paper we focalize on
reversible deterministic
2-counter machines because of their universality property. Indeed,
the following Theorem holds:
\begin{theorem}\cite{Morita}\label{theo:morita}
  For any deterministic Turing machine $\mathcal{T}$ there is a
  deterministic reversible $CM(2)$ $\mathcal{M}$ that simulates
  $\mathcal{T}$.
\end{theorem}
\begin{rem}
  A 2-counter machine $\mathcal{M}=(Q,\delta,\iota,f)$ can be
  sketched as a labeled graph $\mathrsfs{G}(\mathcal{M})$ with set of
  vertices $Q$ and labeled edges $\delta$ where we picture
  $(q_1,i,h,q_2)\in\delta$ as an arrow from $q_1$ to $q_2$ labelled by
  $i,h$ with $i\in\{1,2\}$ and $h\in\{a,\perp,+,0,-\}$.
  It is clear from the definition that if $\mathcal{M}$ is
  deterministic then the only case where a vertex $q$ of
  $\mathrsfs{G}(\mathcal{M})$ may have two outgoing edges is when we have
  tests instructions (Fig. \ref{fig1}),
  i.e. $(q,i,a,q_1),(q,i,\perp,q_2)$ are two edges of
  $\mathrsfs{G}(\mathcal{M})$ with $i\in\{1,2\}$. Dually, if
  $\mathcal{M}$ is reversible then the only case where a vertex $q$ of
  $\mathrsfs{G}(\mathcal{M})$ may have two ingoing edges is when they
  are of the form $(q_1,i,a,q),(q_2,i,\perp,q)$ for some $i\in\{1,2\}$
  (Fig. \ref{fig1}).
\end{rem}
\begin{figure}
 \begin{center}
 \includegraphics[scale=0.7]{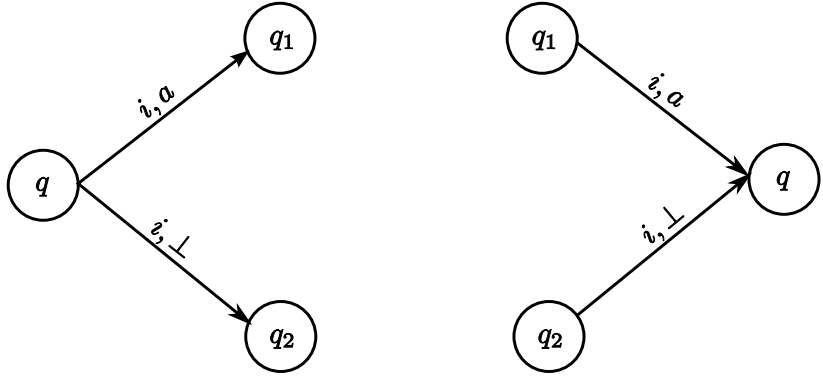}
 \end{center}
  \caption{The deterministic case (on the left) and the reversible case
    (on the right).}\label{fig1}
\end{figure}
For our purpose we are interested in deterministic 2-counter
machines with the following property.
\begin{defn}[alternating $2$-counter machine]
  Let $\mathcal{M}=(Q,\delta,\iota,f)$ be a deterministic $CM(2)$.
  We say that $\mathcal{M}$ is alternating if, for all pairs of
  different instructions $(q,i,h,q'),(q',j,h',q'')\in\delta$, we have
  $j= 3-i$.
\end{defn}
We can prove the following proposition:
\begin{prop}\label{alter}
  Let $\mathcal{M}=(Q,\delta,\iota,f)$ be a deterministic reversible
  $CM(2)$. Then there is a deterministic reversible and alternating
  $CM(2)$ $\mathcal{M}'$ that simulates $\mathcal{M}$.
\end{prop}
\begin{proof}
  What we actually do is to add some dummy states whenever there occur
  two instructions $(p,i,h,q),(q,i,x_1,r_1)\in\delta$. Formally
  $\mathcal{M}'=(Q',\delta',\iota,f)$ is obtained
  applying iteratively the following procedure. Whenever
 $(p,i,h,q)$, $(q,i,x_1,r_1)\in\delta$, we substitute in
  $\delta$ all the instructions $(q,i,x_k,r_k)\in\delta$ (at most
  two since $\mathcal{M}$ is deterministic) by the pairs of
  instructions
  $$
  (q,3-i,0,q_{(q,i,x_k,r_k)}), (q_{(q,i,x_k,r_k)},i,x_k,r_k)
  $$
  where $q_{(q,i,x_k,r_k)}$ is the dummy state added. It is easy to
  see that the obtained $CM(2)$ $\mathcal{M}'$ is an alternating
  deterministic and reversible $CM(2)$ that simulates $\mathcal{M}$.
\end{proof}

Note that we can always replace an instruction $(p,i,0,q)$ by the
couple of instructions $(p,i,a,q), (p,i,\perp,q)$ and on doing so the
$CM(2)$ remains deterministic, reversible and alternating. Therefore,
we shall assume henceforth that our $CM(2)$ have no instructions of
type $(p,i,0,q)$. A $CM(2)$ which is deterministic, reversible, alternating and has no instructions of the form type $(p,i,0,q)$ is said to be {\em normalized}. Taking in particular a universal Turing machine in Theorem \ref{theo:morita} and being undecidable whether or not a universal Turing machine can accept a given input, it follows that there exists a $CM(2)$
$\mathcal{M}_0$ such that it is undecidable whether or not a given
$(m,n) \in \mathbb{N}^2$ is accepted by $\mathcal{M}_0$. Therefore by
Proposition \ref{alter} we have the following Corollary.
\begin{cor}\label{cor:universal}
  There exists a normalized $CM(2)$ $\mathcal{M}^*$ such that it is
  undecidable whether or not a given $(m,n) \in \mathbb{N}^2$ is
  accepted by $\mathcal{M}^*$.
\end{cor}

\section{Amalgam associated to a $2$-counter machine}
In this section we associate an amalgam $[S_1,S_2;U]$ of inverse
semigroups to a deterministic reversible alternating $2$-counter
machine
$$
\mathcal{M}=(Q,\delta,\iota,f)
$$
The rough idea is depicted in Fig. \ref{fig2}: we encode the two tapes
of the machine $\mathcal{M}$ by two inverse semigroups $S_1, S_2$ and the
control of $\mathcal{M}$ is handled through a common inverse subsemigroup
$U$.
\begin{figure}
 \begin{center}
 \includegraphics[scale=0.6]{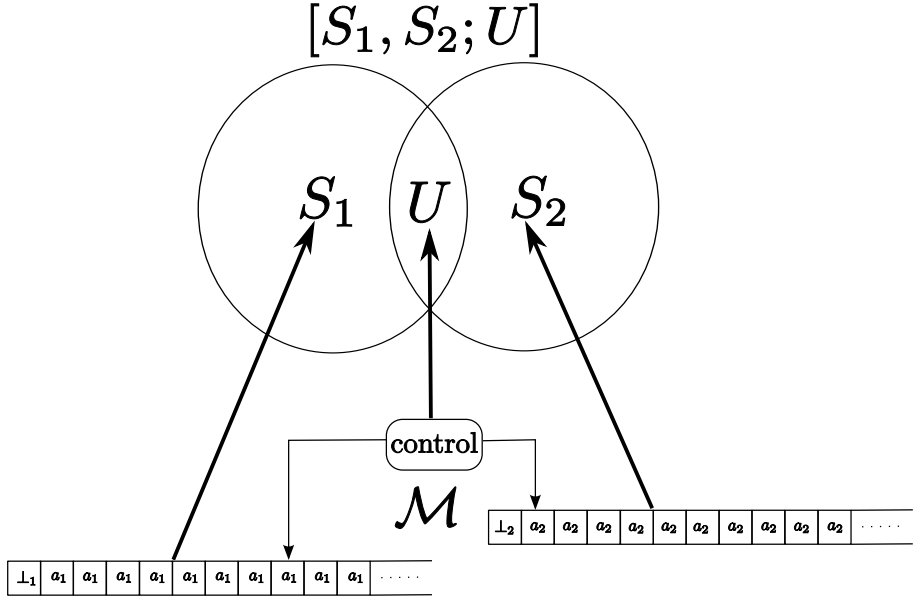}
 \end{center}
  \caption{The rough idea of the encoding.}\label{fig2}
\end{figure}
\\
We start by associating to $\mathcal{M}$ two inverse semigroups $S_1, S_2$,
representing the two tapes, called respectively the \emph{left tape
  inverse semigroup} and the \emph{right tape inverse semigroup} of
$\mathcal{M}$.
\begin{defn}[left and right tape inverse semigroups]\label{def:tape
    inverse semigroup}
Let $\mathcal{M}=(Q,\delta,\iota,f)$ be a deterministic reversible
alternating $2$-counter machine. The left and right tape inverse
semigroups associated to $\mathcal{M}$ are the inverse semigroups
$S_i$ $(i= 1,2)$ presented by Inv$\la X_i \mid \mathcal{T}_i\ra$, where:
$$
X_i=\{\perp_i, a_i,t_i\} \cup \{q_i:q\in Q\},
$$
$$
\mathcal{T}_i=\mathcal{T}_i^c\cup\mathcal{T}_i^t\cup\mathcal{T}_i^w
\cup\mathcal{T}_i^e\cup\mathcal{T}_i^f
$$
and:
\begin{itemize}
\item
$\mathcal{T}_i^c$ are the commuting relations, used to keep track of
the instantaneous description of the machine;
\item
$\mathcal{T}_i^t$ are the test relations (corresponding
to the test instructions);
\item
$\mathcal{T}_i^{w}$ are the writing relations (corresponding
to instructions that move the head of the $i$th tape to the right);
\item
$\mathcal{T}_i^{e}$ are the erasing relations (corresponding
to instructions that move the head of the $i$th tape to the left);
\item
$\mathcal{T}_i^f$ are the finiteness relations (designed to enforce some
finiteness properties on the semigroup $S_i$).
\end{itemize}
More precisely the relations are defined in the following way:
\begin{itemize}
\item[(i)]
$\mathcal{T}_i^d$ consists of all the relations of the form $t_ix
=xt_i$, for $x\in\{a_i,a_i\inv,\perp_i,\perp_i\inv \}$ and $i = 1,2$.
\item[(ii)]
For each test instruction $(p,1,s,q)\in\delta$
$(s\in A)$, we add to $\mathcal{T}^t_1$ the relation
$sp_1=st_1q_1t_1^{-1}$.
\item[(iii)]
For each test instruction $(p,2,s,q)\in\delta$
$(s\in A)$, we add to $\mathcal{T}^t_2$ the relation
$p_2s=t_2q_2t_2^{-1}s$.
\item[(iv)]
For each right move instruction $(p,1,+,q)\in\delta$, we add to
$\mathcal{T}^w_1$
the relations $sp_1=st_1a_1q_1t_1^{-1}$ $(s\in A)$.
\item[(v)]
For each right move instruction $(p,2,+,q)\in\delta$, we add to
$\mathcal{T}^w_2$
the relations $p_2s=t_2q_2a_2t_2^{-1}s$ $(s\in A)$.
\item[(vi)]
For each left move instruction $(p,1,-,q)\in\delta$, we add to
$\mathcal{T}^w_1$
the relations $sa_1p_1=st_1q_1t_1^{-1}$ $(s\in A)$.
\item[(vii)]
For each left move instruction $(p,2,-,q)\in\delta$, we add to
$\mathcal{T}^w_2$
the relations $p_2a_2s=t_2q_2t_2^{-1}s$ $(s\in A)$.
\item[(viii)] The set of relations $\mathcal{T}_1^f$ is formed by
$$\begin{array}{l}
f_1x = xf_1 = f_1 \quad (x \in \wt{X_1}),\\
q_1x = f_1  \quad (q \in Q, \; x \in \wt{X_1}\setminus \{
q_1\inv\}),\\
p_1\inv q_1 = f_1 \quad (p,q \in Q \mbox{ distinct}),\\
a_1\perp_1\inv = \perp_1\inv a_1 = f_1.
\end{array}$$
\item[(ix)] The set of relations $\mathcal{T}_2^f$ is formed by
$$\begin{array}{l}
f_2x = xf_2 = f_2 \quad (x \in \wt{X_2}),\\
xq_2 = f_2  \quad (q \in Q, \; x \in \wt{X_2}\setminus \{
q_2\inv\}),\\
p_2 q_2\inv = f_2 \quad (p,q \in Q \mbox{ distinct}),\\
a_2\perp_2\inv = \perp_2\inv a_2 = f_2.
\end{array}$$
\end{itemize}
\end{defn}

\begin{rem}\label{rem:one to one encoding}
Clearly $f_i$ is a zero element for the inverse semigroup $S_i$.
Note also that in cases (iv)-(vii) we have added two relations for a
single move to
force $a_i$ or $\perp_i$ to be present on both sides of each
relation. This is done for reasons
which will become clear when we shall introduce the inverse semigroup
$U$. Indeed, this way the embeddability of $U$ into the tape inverse
semigroups is verified since we avoid any nontrivial relations having
only generators of $U$ in one side. On the other hand, each test
instruction gives birth to a single relation (cases (ii)-(iii)).
\end{rem}

We associate to $\mathcal{M}$ another inverse semigroup $U$ which
represents its control unit:

\begin{defn}[core inverse semigroup]\label{def:core inverse semigroup}
  Let $\mathcal{M}=(Q,\delta,\iota,f)$ be a deterministic reversible
  alternating $2$-counter machine. The core inverse semigroup
  of $\mathcal{M}$ is the inverse semigroup $U$ presented by Inv$\la
  X_U\mid \mathcal{T}_U\ra$, where $X_U = Q \cup \{ t \}$ and the
set of relations $\mathcal{T}_U$ is formed by
$$\begin{array}{l}
fx = xf = f \quad (x \in \wt{X_U}),\\
pq = f  \quad (p,q \in Q),\\
pq\inv = p\inv q = f \quad (p,q \in Q \mbox{ distinct}).
\end{array}$$
\end{defn}
Similarly to the previous cases, $U$ has a zero $f$. With the notation of Definitions \ref{def:tape inverse semigroup} and \ref{def:core inverse semigroup} we can prove the following result:
\begin{prop}\label{prop:embedding property}
  Let $\mathcal{M}$ be a deterministic reversible alternating
  $2$-counter machine and let $S_1,S_2,U$ be respectively the
  left-right tape inverse semigroups and the core inverse semigroup
  of $\mathcal{M}$. The map $\omega_i$ defined by
  $$
  \omega_i(t)=t_i,\;\omega_i(q)=q_i \; (q \in Q)
  $$
  can be extended to a monomorphism $\omega_i:U\hookrightarrow S_i$
  for $i=1,2$.
\end{prop}
\begin{proof}
Let $u,v \in \wt{X_U}^+$ and assume that $\omega_i(u) = \omega_i(v)$
in $S_i$. Suppose first that  $\omega_i(u)$ is not the zero element of
$S_i$. Then we are forbidden to use relations from $\mathcal{T}_i^f$ to
transform  $\omega_i(u)$ into $\omega_i(v)$. Since all the other
relations of $\mathcal{T}_i$ involve $a_i$ or $\perp_i$, it follows
that $\omega_i(v)$ must actually equal $\omega_i(u)$ in the free
inverse semigroup over $X_U$, and so we definitely have $u = v$ in
$U$. Hence we may assume that $\omega_i(u) = f_i$ in $S_i$. But the unique
relations in $\mathcal{T}_i$ where neither $a_i$ nor $\perp_i$ occur
are precisely (up to subscript $i$) those of $\mathcal{T}_U$, hence we
may apply one of the latter relations to $u$ (possibly after
some Vagner congruence relations) to get an occurrence of $f$.
Thus $u = f$ in $U$. Similarly, also $v = f$ in $U$ and so $u = v$
holds in $U$ as required.
\end{proof}

In view of Proposition \ref{prop:embedding property} we can associate
to a deterministic reversible alternating $2$-counter machine an
amalgam:

\begin{defn}\label{defn:amalgam associate to M}
  Let $\mathcal{M}=(Q,\delta,\iota,f)$ be a deterministic reversible
  alternating $2$-counter machine. The amalgam of inverse
  semigroups associated to $\mathcal{M}$ is the $5$-tuple
  $[S_1,S_2;U,\omega_1,\omega_2]$ where $S_1, S_2$ are the left-right
  tape inverse semigroups of $\mathcal{M}$, $U$ is the core
  inverse semigroup of $\mathcal{M}$ and
  $\omega_i:U\hookrightarrow S_i$ are the embeddings of Proposition
  \ref{prop:embedding property}. In this way the free product with
  amalgamation of the amalgam $[S_1,S_2;U,\omega_1,\omega_2]$
  associated to $\mathcal{M}$ can be presented by
  $$\mbox{Inv}
  \la X_1\cup X_2\mid \mathcal{T}_1\cup\mathcal{T}_2\cup\mathcal{T}_3\ra
  $$
  where
  $$
  \mathcal{T}_3=\{(q_1,q_2):q\in Q\}\cup\{(t_1,t_2)\}
  $$
\end{defn}

The left-right tape inverse semigroups $S_i$ have the following
important property:

\begin{prop}
\label{finiter}
  Let $\mathcal{M}$ be a deterministic reversible alternating
  $2$-counter machine and let $S_1,S_2$ be respectively the left-right
  tape inverse semigroups of $\mathcal{M}$. Then the Green
  $\mathcal{R}$-classes of $S_i$ are finite for $i = 1,2$.
\end{prop}
\begin{proof}
First of all, let $\mathcal{M}'$ be the machine $\mathcal{M}$ with
counters 1 and 2
reversed. It is easy to check that $t_2 \to t_1\inv$, $q_2 \to q_1$,
$a_2 \to a_1$, $\perp_2 \to \perp_1$ induces an isomorphism from $S_2$
onto the dual of $S'_1$, hence it suffices to show that the Green
$\mathcal{R}$-classes of $S_1$ are finite. This is of course
equivalent to say that the \sch graphs of $S_1$ are finite. To
simplify notation, we drop all the subscripts 1 in $S_1$, $\mathcal{T}_1$,
$X_1$ and its letters.

Let $w \in \wt{X}^+$ and let $(\mathcal{A}_k)_k$ be the Stephen's
sequence whose colimit is
$\mathcal{A}(X,\mathcal{T};w)$. We can of course
assume that $w$ does not represent the zero element of $S$, which is
the single element of its $\mathcal{R}$-class. Let $m_0$ be the
total number of edges labeled by letters of $Q$ in the Munn tree of
$w$. We claim that there are at most $m_0$ expansions in the sequence
$(\mathcal{A}_k)_k$ featuring edges labeled by elements of $Q$
($Q$-edges).

Indeed, let $Q'$ consist of all $p \in Q$ such that there exists some
instruction $(p,1,\ldots, \ldots) \in \delta$, and let $Q''$ consist
of all $q \in
Q$ such that there exists some instruction $(\ldots,1,\ldots, q) \in
\delta$. Since $\mathcal{M}$ is alternating, we have $Q' \cap Q'' =
\emptyset$. Moreover, since $\mathcal{M}$ is deterministic
(respectively reversible), each element of $Q'$ (respectively $Q''$)
can feature a unique expansion: this is true for relations of type
(ii), which are the unique outcome of a test instruction, and even if
two test instructions $(p,1,a,q), (p,1,\perp,r)$ are present in
$\delta$, the corresponding relations can never be used simultaneously for the
same $p$-edge in view of the forbidden relations $a\perp\inv = f$ (we
are assuming that $w \neq f$ in $S$). Similarly, the same $Q$-edge
cannot feature two expansions involving relations of type (iv) or
(vi), where application of both relations arising from a single
instruction would ensure the appearance of a factor $a\perp\inv$. And
of course, the relations of (viii) are altogether forbidden since $w
\neq f$ in $S$.

Therefore there exists some $m_1 \in \mathbb{N}$ such that all expansions and
foldings involving $Q$-edges in $(\mathcal{A}_k)_k$ have taken place
before we reach $\mathcal{A}_{m_1}$.

Given $Y \subseteq X$, let $\sigma_Y:\wt{X}^* \to \mathbb{Z}$ be the
homomorphism defined by
$$\sigma_Y(x) = \left\{
\begin{array}{rl}
1&\mbox{ if }x \in Y\\
-1&\mbox{ if }x \in Y\inv \\
0&\mbox{ if }x \in \wt{X} \setminus \wt{Y}
\end{array}
\right.$$
We claim that
\begin{equation}
\label{tloop}
\mbox{If $u$ labels a loop in some $\mathcal{A}_k$, then $\sigma_t(u) =
  0$.}
\end{equation}
Indeed, this holds trivially for the Munn tree $\mathcal{A}_1$. Since
$\sigma_t(r) = \sigma_t(r')$ for every $(r,r') \in \mathcal{T} \setminus
\mathcal{T}^f$, this property is inherited through expansions in the
Stephen's sequence (expansions featuring relations of $\mathcal{T}^f$
are forbidden since $w \neq f$ in $S$). Finally, the property is
inherited through foldings as well: if $\mathcal{A}'$ is obtained from
$\mathcal{A}$ folding two edges of label $x$, then every loop in
$\mathcal{A}'$ can be lifted to a loop in $\mathcal{A}$ by
successively inserting factors of the form $x\inv x$. Therefore
(\ref{tloop}) holds. We complete now the proof of our proposition by showing that the
$\mathcal{T}^c$-closure of $\mathcal{A}_{m_1}$ is finite. Note that
the $\mathcal{T}^c$-closure of $\mathcal{A}_{m_1}$ is also its
$\mathcal{T}$-closure: we cannot use relations of $\mathcal{T}^t \cup
\mathcal{T}^w \cup \mathcal{T}^e$ by definition of $m_1$, and we
cannot use relations of $\mathcal{T}^f$ since $u \neq f$ in $S$.

Let $C_1, \ldots, C_{\ell}$ be the connected components of the automaton
obtained by removing all $Q$-edges from $\mathcal{A}_{m_1}$ (the
$c$-components). Clearly,
performing $\mathcal{T}^c$-expansions cannot merge $c$-components,
neither can folding edges with label in $\wt{A \cup t}$ (and folding
of $Q$-edges will not occur after $\mathcal{A}_{m_1}$. Therefore we
only need to prove that the $\mathcal{T}^c$-closure of a $c$-component
$C_j$ is finite for $j = 1,\ldots, \ell$. In view of (\ref{tloop}), it is enough
to prove that
\begin{itemize}
\item[(P)]
If $\mathcal{A}$ is a finite inverse $\wt{A\cup t}$-automaton where
every $u \in \wt{A\cup t}^*$ labelling a loop satisfies $\sigma_t(u) =
0$, then the $\mathcal{T}^c$-closure $\mathcal{B}$ of $\mathcal{A}$ is finite.
\end{itemize}
Let $n = |V(\mathcal{A})|$ and suppose that there is a path $i \mapright{v} j$
in $\mathcal{A}$ with $|\sigma_t(v)| \geq n$. We may assume that $v$ is
shortest possible. Since $|v| \geq n$, the path $i \mapright{v} j$
must contain a loop, hence we may factor it as
$$i \mapright{v_1} i_1 \mapright{u} i_1 \mapright{v_2} j$$
with $v = v_1uv_2$ and $u \neq 1$. Now $\sigma_t(u) = 0$ yields
$\sigma_t(v_1v_2) = \sigma_t(u)$ and so $v_1v_2$ is a shorter
alternative to $u$, a contradiction. Thus
$|\sigma_t(v)| < n$ whenever $v$ labels a path in $\mathcal{A}$.
We claim that also
\begin{equation}
\label{tbound}
|\sigma_t(v)| < n \mbox{ whenever $v$ labels a path in } \mathcal{B}.
\end{equation}
Indeed, it suffices to remark that this property is inherited through
$\mathcal{T}^c$-expansions and foldings, using the same argument in the
proof of (\ref{tloop}).

Assume that $i_1, \ldots, i_r$ are the vertices of $\mathcal{B}$
corresponding to the starting points of edges labelled by $a$ or
$\perp$ in $\mathcal{A}$. We claim that if $i \mapright{s} j$ is an
edge of $\mathcal{B}$ with $s \in A$, then there exists a path $i_k
\mapright{v} i$ in $\mathcal{B}$ for some $k \in \{ 1, \ldots, r\}$
and $v \in \wt{t}^*$. Once again, it suffices to check that this
property is inherited through $\mathcal{T}^c$-expansions and foldings,
a quite trivial observation. Since $\mathcal{B}$ is inverse, we may in
fact assume that $v \in t^* \cup (t\inv)^*$. In view of
(\ref{tbound}), it follows that there are only finitely many $A$-edges
in $\mathcal{B}$, and again by (\ref{tbound}), the number of $t$-edges
must also be finite. Therefore (P) holds and the proposition is proved.
\end{proof}

\section{Undecidability of the word problem for amalgams}

In this section we prove that the word problem for an amalgam\linebreak
$[S_1,S_2;U,\omega_1,\omega_2]$ of inverse
semigroups is undecidable even if we assume $S_1$ and $S_2$ (and
therefore $U$) to have
finite $\mathcal{R}$-classes and $\omega_1,\omega_2$ to be
computable functions. The idea is to relate the computation of a
deterministic reversible alternating $2$-counter machine with the \sch
graphs of the associated amalgam.

We fix now a normalized $CM(2)$ $\mathcal{M} = (Q,\delta,\iota,f)$ and consider
the semigroup $S$ defined by the associated amalgam
$[S_1,S_2;U,\omega_1,\omega_2]$. Write $X = X_1\cup X_2$ and
$\mathcal{T} = \mathcal{T}_1\cup\mathcal{T}_2\cup\mathcal{T}_3$.
Clearly $f_1$ and $f_2$ represent both the
zero of $S$. Let $(m,n) \in \mathbb{N}^2$ and write $w_{m,n} =
\perp_1a_1^m\iota_1 a_2^n\perp_2$. Suppose that
$$(\iota,m,n) = (q^{(0)},m_0,n_0) \vdash_{\mathcal{M}} \ldots
\vdash_{\mathcal{M}} (q^{(k)},m_k,n_k).$$
Since $\mathcal{M}$ is deterministic, there is at most one such
sequence of a given length starting with $(\iota,m,n)$.
Write $m'_k = \max\{ m_0, \ldots, m_k\}$, $n'_k = \max\{ n_0, \ldots,
n_k\}$. We define a finite inverse $X$-automaton
$\mathcal{B}_{m,n}^{(k)}$ as follows (describing only the edges with
positive label):
\begin{itemize}
\item
The vertices are of the form $c_{i,j}$ and $d_{i,\ell}$ for $i =
0,\ldots,k$ and $j = 0,\ldots,m'_k+1$ and $\ell = 0,\ldots,n'_k+1$.
\item
$c_{0,0}$ is initial and $d_{0,0}$ is final.
\item
There exist edges $c_{i-1,j} \longmapright{t_1,t_2} c_{i,j}$ for all $i =
1,\ldots,k$ and $j = 0,\ldots,m'_k+1$.
\item
There exist edges $d_{i-1,\ell} \longmapright{t_1,t_2} d_{i,\ell}$ for all $i =
1,\ldots,k$ and $\ell = 0,\ldots,n'_k+1$.
\item
There exist edges $c_{i,0} \mapright{\perp_1} c_{i,1}$ for all $i =
0,\ldots,k$.
\item
There exist edges $c_{i,j} \mapright{a_1} c_{i,j+1}$ for all $i =
0,\ldots,k$ and $j = 1,\ldots,m'_k$.
\item
There exist edges $d_{i,1} \mapright{\perp_2} d_{i,0}$ for all $i =
0,\ldots,k$.
\item
There exist edges $d_{i,j+1} \mapright{a_2} d_{i,j}$ for all $i =
0,\ldots,k$ and $j = 1,\ldots,n'_k$.
\item
There exist edges $c_{i,m_i+1} \vlongmapright{q_1^{(i)},q_2^{(i)}}
d_{i,n_i+1}$ for all $i = 0,\ldots,k$.
\end{itemize}

\begin{figure}
 \begin{center}
 \includegraphics[scale=0.8]{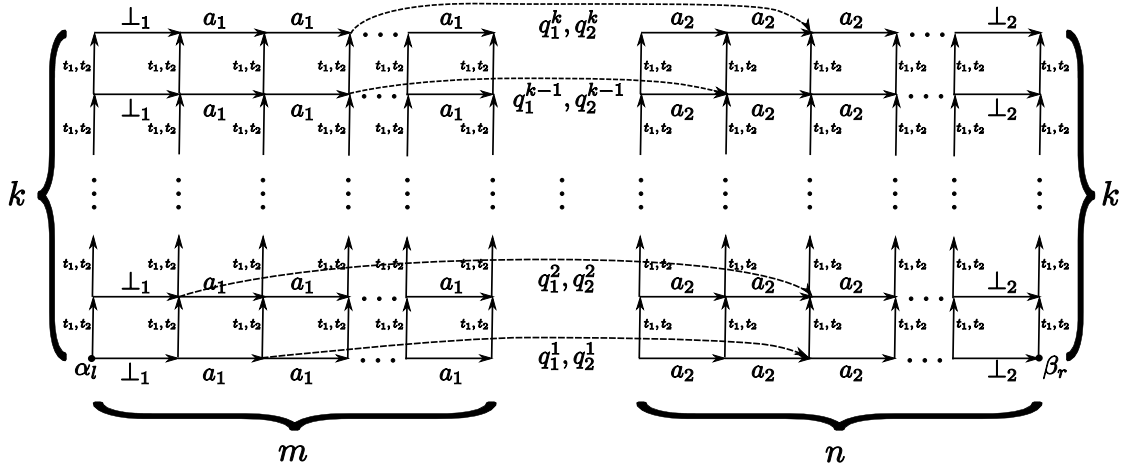}
 \end{center}
  \caption{The automaton $\mathcal{B}_{m,n}^{(k)}$}\label{fig3}\end{figure}
\begin{lemma}
\label{config}
Let $\mathcal{M}$ be a normalized $CM(2)$ and let
$m,n,k \in \mathbb{N}$. Then
$\mathcal{B}_{m,n}^{(k)}$ is a finite approximation of
$\mathcal{A}(X,\mathcal{T};w_{m,n})$.
\end{lemma}

\begin{proof}
We use induction on $k$. Clearly, $\mathcal{B}_{m,n}^{(0)}$ is the
Munn tree of $w_{m,n}$ with the edge $c_{0,m+1} \mapright{\iota_2}
d_{0,n+1}$ adjoined. Since $\iota_1 = \iota_2$ is a relation of
$\mathcal{T}$, the lemma holds for $k = 0$.

Assume now that $k > 0$ and $\mathcal{B}_{m,n}^{(k-1)}$ is a finite
approximation of the \sch automaton $\mathcal{A}(X,\mathcal{T};w_{m,n})$. Assume that
$$(\iota,m,n) \vdash_{\mathcal{M}}^{k-1} (q^{(k-1)},m_{k-1},n_{k-1})
\vdash_{\mathcal{M}} (q^{(k)},m_k,n_k).$$
Then $(q^{(k-1)},i,x,q^{(k)}) \in \delta$ for some $i \in \{ 1,2\}$
(assume $i = 1$ for simplicity, the case $i = 2$ being analogous)
and
$$x = \left\{
\begin{array}{ll}
a&\mbox{ if }m_{k} = m_{k-1} > 0\\
\perp&\mbox{ if }m_{k} = m_{k-1} = 0\\
+&\mbox{ if }m_{k} = m_{k-1} +1\\
-&\mbox{ if }m_{k} = m_{k-1} -1
\end{array}
\right.$$
The instruction $(q^{(k-1)},1,x,q^{(k)})$ produces a relation
$syq_1^{(k-1)} = st_1zq_1^{(k)}t_1\inv$ in
$\mathcal{T}_1^t \cup \mathcal{T}_1^w \cup \mathcal{T}_1^e$ with $y =
a_1$ if $x = -$ and $y = 1$ otherwise, and  $z =
a_1$ if $x = +$ and $z = 1$ otherwise. Using $s = \perp_1$ or $a_1$
according to the values of $m_{k-1}$ and $y$, we can now use this
relation to perform an expansion of the path
$$
c_{k-1, m_{k-1}-|y|}\vlongmapright{syq_1^{(k-1)}} d_{k-1, n_{k-1}+1}
$$
to get after subsequent folding the edges
$$c_{k,m_k+1}\mapright{q_1^{(k)}}d_{k,n_k+1}, c_{k-1,m_k+\varepsilon}\mapright{t_1}c_{k,m_k+\varepsilon}
$$
for $\varepsilon = 0$ or $1$, $d_{k-1,n_k+1} \mapright{t_1} d_{k,n_k+1}$
and possibly also an edge of the form $c_{k,m_k} \mapright{a_1}
c_{k,m_k+1}$. It is straightforward to check that, applying now the
relations from $\mathcal{T}_1^c \cup
\mathcal{T}_2^c$, followed by $t_1 = t_2$ and $q_1^{(k)} = q_2^{(k)}$,
we get precisely $\mathcal{B}_{m,n}^{(k)}$. Note that if $m_k
=m'_{k-1} +1$, we get an edge $c_{k,m'_{k-1}+1} \mapright{a_1}
c_{k,m'_k+1}$ through the expansion, and the
application of the relations from $\mathcal{T}_1^c$
yields indeed the required extra column of $\mathcal{B}_{m,n}^{(k)}$
with respect to $\mathcal{B}_{m,n}^{(k-1)}$. Therefore
$\mathcal{B}_{m,n}^{(k)}$ is a finite approximation of
$\mathcal{A}(X,\mathcal{T};w_{m,n})$ and the lemma is proved.
\end{proof}

It follows from the definition that $\mathcal{B}_{m,n}^{(k-1)}$ embeds
in $\mathcal{B}_{m,n}^{(k)}$ for every $k \geq 1$. Therefore we can
define $\mathcal{B}_{m,n}$ as the colimit of the sequence
$(\mathcal{B}_{m,n}^{(k)})_k$, where all the $\mathcal{B}_{m,n}^{(k)}$
embed. This colimit may be finite or infinite, depending on whether or not
the computation in $\mathcal{M}$ halts when we start with the
configuration $(\iota,m,n)$.

Let $\mathcal{C}$ denote the finite complete inverse $X$-automaton
with a single vertex.

\begin{prop}
\label{computsch}
Let $\mathcal{M}$ be a normalized $CM(2)$ and let
$m,n \in \mathbb{N}$. Then
$$\mathcal{A}(X,\mathcal{T};w_{m,n}) = \left\{
\begin{array}{ll}
\mathcal{C} & \mbox{ if $(m,n)$ is accepted by }\mathcal{M}\\
\mathcal{B}_{m,n} & \mbox{otherwise}
\end{array}
\right.$$
\end{prop}

\begin{proof}
First of all, we claim that $\mathcal{B}_{m,n}$ is closed under all
relations of $\mathcal{T} \setminus (\mathcal{T}_1^f \cup
\mathcal{T}_2^f)$. This is immediate for all the relations in
$\mathcal{T}_1^c \cup \mathcal{T}_2^c \cup \mathcal{T}_3$. Since
$\mathcal{M}$ is deterministic, reversible and alternating, it is also
true for relations in $\mathcal{T}_i^t \cup \mathcal{T}_i^w \cup
\mathcal{T}_i^e$: no $Q$-edge can be involved in more than one
expansion. For instance, reversibility prevents the appearance in line
$i-1$ of some new $Q$-edge obtained through an expansion of a $Q$-edge
in line $i$.

Now, if $(m,n)$ is accepted by $\mathcal{M}$, then $f$ labels some
edge in $\mathcal{B}^{(k)}_{m,n}$ for some $k$ and so the relations of
$\mathcal{T}_1^f \cup \mathcal{T}_2^f$ eventually collapse
$\mathcal{B}^{(k)}_{m,n}$ into $\mathcal{C}$. Since
$\mathcal{B}^{(k)}_{m,n}$ is a finite approximation of
$\mathcal{A}(X,\mathcal{T};w_{m,n})$, it follows that
$\mathcal{A}(X,\mathcal{T};w_{m,n}) = \mathcal{C}$.

Finally, assume that $(m,n)$ is not accepted by $\mathcal{M}$. It is
straightforward to check that in this case $\mathcal{B}_{m,n}$ must be
also closed under $\mathcal{T}_1^f \cup
\mathcal{T}_2^f$, hence $\mathcal{B}_{m,n}$ is
$\mathcal{T}$-closed. Since
$\mathcal{B}^{(k)}_{m,n}$ is a finite approximation of
$\mathcal{A}(X,\mathcal{T};w_{m,n})$, we have $L(\mathcal{B}^{(k)}_{m,n})
\subseteq L(\mathcal{A}(X,\mathcal{T};w_{m,n}))$ for every $k$.
Hence
$$L(\mathcal{B}_{m,n}) = \bigcup_{k \geq 0} L(\mathcal{B}^{(k)}_{m,n})
\subseteq L(\mathcal{A}(X,\mathcal{T};w_{m,n})).$$
Since $\mathcal{A}(X,\mathcal{T};w_{m,n})$ is the smallest $\mathcal{T}$-closed
inverse $X$-automaton recognizing $w_{m,n}$, it follows that
$\mathcal{A}(X,\mathcal{T};w_{m,n}) = \mathcal{B}_{m,n}$ as claimed.
\end{proof}

Since $\mathcal{C}$ is the \sch automaton of the zero of $S$, we
immediately get:

\begin{theorem}
\label{theequiv}
Let $\mathcal{M}$ be a normalized $CM(2)$ and let
$m,n \in \mathbb{N}$. Then $w_{m,n} = 0$ in the associated amalgam
$[S_1,S_2;U,\omega_1,\omega_2]$
if and only if
$(m,n)$ is accepted by $\mathcal{M}$.
\end{theorem}

We can now prove our main result:

\begin{theorem}
The word problem for an amalgam
$[S_1,S_2;U,\omega_1,\omega_2]$ of inverse
semigroups may be undecidable even if we assume $S_1$ and $S_2$ (and
therefore $U$) to have
finite $\mathcal{R}$-classes and $\omega_1,\omega_2$ to be
computable functions.
\end{theorem}

\begin{proof}
If the word problem would be decidable in these circumstances, then,
 in view of Proposition \ref{finiter} and Theorem \ref{theequiv}, we could
decide whether or not a normalized $CM(2)$ accepts a given $(m,n) \in
\mathbb{N}^2$. And the latter is undecidable, even when we consider the single
$CM(2)$ $\mathcal{M}^*$ of Corollary \ref{cor:universal}.
\end{proof}
As a consequence of Theorem \ref{theequiv} we have also the following:
\begin{theorem}
  Checking finiteness for a $\mathcal{D}$-class of an amalgam
  $[S_1,S_2;U,\omega_1,\omega_2]$ of
  inverse semigroups may be undecidable even if we assume $S_1$ and
  $S_2$ (and therefore $U$) to have finite $\mathcal{D}$-classes.
\end{theorem}
\begin{proof}
Clearly, checking the finiteness of a $\mathcal{D}$-class of a word $w$
in $S_1*_U S_2$ is equivalent to check whether or not the
corresponding \sch automaton is finite. Consider again the $CM(2)$
$\mathcal{M}^*$ of Corollary \ref{cor:universal} and the associated
amalgam. By Proposition
\ref{computsch}, if
$\mathcal{A}(X,\mathcal{T};w_{m,n})$ would be infinite, then $(m,n)$
is not accepted by $\mathcal{M}^*$. On the other hand, if
$\mathcal{A}(X,\mathcal{T};w_{m,n})$ would be finite, we could use the
Stephen's sequence to actually compute it \cite{Steph} and decide
whether or not $\mathcal{M}^*$ accepts $(m,n)$. Therefore we cannot
decide whether or not $\mathcal{A}(X,\mathcal{T};w_{m,n})$ is finite.
\end{proof}

\section*{Acknowledgments}
Both authors acknowledge support from C.M.U.P., financed by F.C.T.
(Portugal) through the programmes POCTI and POSI, with national and
E.U. structural funds. The first-named author acknowledges also the support of the FCT project SFRH/BPD/65428/2009.

\bibliography{biblio}

\begin{thebibliography}{10}

\bibitem{ChNuRo}
A.Cherubini, C.Nuccio, and E.~Rodaro.
\newblock Multilinear equations in amalgams of finite inverse semigroups.
\newblock {\em IJAC}, 21(1-2):35--59, 2011.

\bibitem{Ben}
P.~Bennett.
\newblock Amalgamated free product of inverse semigroups.
\newblock {\em Journal of Algebra}, 198:499--537, 1997.

\bibitem{BiMaMi}
J.C. Birget, S.W. Margolis, and J.C. Meakin.
\newblock On the word problem for tensor products and amalgams of monoids.
\newblock {\em Internat. J. Algebra Comput.}, 9(3-4):271–--294, 1999.

\bibitem{Finite}
A.~Cherubini, J.~Meakin, and B.~Piochi.
\newblock Amalgams of finite inverse semigroups.
\newblock {\em Journal of Algebra}, 285:706--725, 2005.

\bibitem{Hall2}
T.E. Hall.
\newblock Free products with amalgamation of inverse semigroups.
\newblock {\em J. Algebra}, 34:375–--385, 1975.

\bibitem{Kim}
N.~Kimura.
\newblock On semigroups.
\newblock {\em PhD Thesis at Tulane University of Louisiana}, 1957.

\bibitem{lyndon}
R.C. Lyndon and P.E. Schupp.
\newblock {\em Combinatorial Group Theory}.
\newblock Springer, 1977.

\bibitem{Morita}
K.~Morita.
\newblock Universality of a reversible two-counter machine.
\newblock {\em Theoretical Computer Science}, 168:303--320, 1996.

\bibitem{munn}
W.D. Munn.
\newblock Free inverse semigroups.
\newblock {\em Proc. London Math. Soc.}, (3) 29:385--404, 1974.

\bibitem{RoByci}
E.~Rodaro.
\newblock Bicyclic subsemigroups in amalgams of finite inverse semigroups.
\newblock {\em IJAC}, 20(1):89--113, 2010.

\bibitem{RoChe}
E.~Rodaro and A.~Cherubini.
\newblock Decidability of word problem in \textsc{Y}amamura's \textsc{HNN}
  extensions of finite inverse semigroups.
\newblock {\em Semigroup Forum}, 77 (2):163--186, 2008.

\bibitem{Sapir}
M.V. Sapir.
\newblock Algorithmic problems for amalgams of finite semigroups.
\newblock {\em Journal of Algebra}, 229 (2):514--531, 2000.

\bibitem{Steph}
J.B. Stephen.
\newblock Presentation of inverse monoids.
\newblock {\em Journal of Pure and Applied Algebra}, 198:81--112, 1990.

\bibitem{Steph98}
J.B. Stephen.
\newblock Amalgamated free products of inverse semigroups.
\newblock {\em Journal of Algebra}, 208:339--424, 1998.

\end{thebibliography}
\bibliographystyle{plain}

\end{document}